\documentclass[oneside,english]{amsart}
\usepackage[T1]{fontenc}
\usepackage[latin1]{inputenc}
\usepackage{indentfirst}
\usepackage{amsmath}
\usepackage{yhmath} 
\usepackage{amsfonts}
\usepackage{amssymb}
\usepackage{textcomp} 
\usepackage{epic}
\usepackage{verbatim} 

\theoremstyle{definition}

\newtheorem{defn}{{Definition}}[section]
\newtheorem*{defn*}{{Definition}}

\newtheorem{que}{{Question}}
\newtheorem*{Q*}{{Question}}

\theoremstyle{plain}
\newtheorem{thm}[defn]{Theorem}
\newtheorem*{thm*}{Theorem}
\newtheorem{cor}[defn]{{Corollary}}
\newtheorem*{cor*}{Corollaire}

\newtheorem{prop}[defn]{{Proposition}}
\newtheorem*{prop*}{{Proposition}}
\newtheorem{lem}[defn]{{Lemma}}
\newtheorem{fact}[defn]{{Fact}}

\theoremstyle{remark}
\newtheorem{rmk}[defn]{{Remark}}

\newcommand{\ima}{\mathrm{Im}}
\renewcommand{\ker}{\mathrm{Ker}}

\begin{document}

\title{Small skew fields}
\author{C\'edric Milliet}
\curraddr[C\'edric Milliet]{Universit\'e de Lyon, Universit\'e Lyon 1, Institut Camille Jordan UMR 5208 CNRS,
43 boulevard du 11 novembre 1918,\newline%
\indent 69622 Villeurbanne Cedex, France}%
\email[C\'edric Milliet]{milliet@math.univ-lyon1.fr}%

\maketitle

\begin{abstract}A small field of positive characteristic is commutative\end{abstract}

Wedderburn showed in 1905 that finite fields are commutative. As for infinite fields, we know that superstable \cite[Cherlin, Shelah]{CS} and supersimple \cite[Pillay, Scanlon, Wagner]{PSW} ones are commutative. In their proof, Cherlin and Shelah use the fact that a superstable field is algebraically closed. Wagner showed that a small field is algebraically closed \cite{W1}, and asked whether a small field should be commutative. We shall answer this question positively in non-zero characteristic.

\section{Preliminaries}

\begin{defn}A theory is \emph{small} if it has countably many $n-$types without parameters for all integer $n$. A structure is \emph{small} if its theory is so.\end{defn}

We shall denote $dcl(A)$ the definable closure of a set $A$. Note that if $K$ is a field and $A$ a subset of $K$, then $dcl(A)$ is a field too. Smallness is clearly preserved under interpretation and addition of finitely many parameters.

Let $D,D_1,D_2$ be $A$-definable sets in some structure ${M}$ with $A\subset M$. We define \emph{the Cantor-Bendixson rank} $CB_A(D)$ and \emph{degree} $dCB_A(D)$ of $D$ over $A$.

\begin{defn}By induction, we define
\begin{itemize}
\item[] $CB_A(D)\geq 0$ if $D$ is not empty
\item[] $CB_A(D)\geq\alpha+1$ if there is an infinite family of disjoint $A$-definable subsets $(D_i)_{i\in\omega}$ of $D$, such that $CB_A(D_i)\geq\alpha$ for all $i<\omega$.

\item[] $CB_A(D)\geq\beta$ for	a limit ordinal $\beta$ if $CB_A(D)\geq\alpha$ for all $\alpha<\beta$.
\end{itemize}\end{defn}

\begin{defn} $dCB_A(D)$ is the greatest integer $d$ such that $D$ can be divided into $d$ disjoint $A$-definable sets, with same rank over $A$ as $D$.\end{defn} 

\begin{prop} If $M$ is small and $A$ is a finite set,
\begin{itemize}
\item[(i)] $CB_A(M)$ is ordinal.
\item[(ii)] $dCB_A$ is well defined.
\item[(iii)] If $D_1\subset D_2$, then $CB_A(D_1)\leq CB_A(D_2)$.
\item[(iv)] $CB_A(D_1\cup D_2)=max\{CB_A(D_1),CB_A(D_2)\}$.
\item[(v)] $CB_A$ and $dCB_A$ are preserved under $A$-definable bijections.
\end{itemize}\end{prop}

If $A$ is empty, we shall write $CB$ and $dCB$ rather than $CB_\emptyset$ or $dCB_\emptyset$.

\begin{rmk}\label{rmq}Let $H< G$ be $A$-definable small groups with $H\cap dcl(A)< G\cap dcl(A)$. Then, either $CB_A(H)<CB_A(G)$, or $CB_A(H)=CB_A(G)$ and $dCB_A(H)<dCB_A(G)$.\end{rmk}

\begin{cor}A small integral domain with unity is a field.\end{cor}

\begin{proof}Let $R$ be this ring. If $r$ is not invertible, then $1\notin rR$ hence $rR\cap dcl(r) \lneq R\cap dcl(r)$. Apply Remark \ref{rmq}, but $R$ and $rR$ have same rank and degree over $r$.\end{proof}

Note that $R$ need not have a unity (see Corollary \ref{cor}). More generally, if $\varphi$ is a definable bijection between two definable groups $A\leq B$ in a small structure, then $A$ equals $B$.

\begin{prop}\emph{(Descending Chain Condition)} Let $G$ be a small group and $\overline{g}$ a finite tuple in $G$. Set $H=\langle\overline{g}\rangle$ or $H=dcl(\overline{g})$. In $H$, there is no strictly decreasing infinite chain of subgroups of the form $G_0\cap H>G_1\cap H>G_2\cap H>\cdots$, where the $G_i$ are $H$-definable subgroups of $G$.\end{prop}

\begin{proof}By Remark 5, either the rank or the degree decreases at each step.\end{proof}

\begin{cor}Let $G$ be a small group, $H<G$ a finitely generated subgroup of $G$, and $(G_i)_{i\in I}$ a family of $H$-definable subgroups of $G$. There is a finite subset $I_0\subset I$ such that $\displaystyle\bigcap_{i\in I}G_i\cap H=\bigcap_{i\in I_0}G_i\cap H$.\end{cor}

Another chain condition on images of endomorphisms :

\begin{prop}\label{prop}Let $G$ be a small group and $h$ a group homomorphism of $G$. There exists some integer $n$ such that $\ima h^n$ equals $\ima h^{n+1}$.\end{prop}

\begin{proof}Suppose that the chain $(\ima h^n)_{n\geq 1}$ be strictly decreasing. Consider the following tree $G(x)$\begin{center} $\ima h$ \\
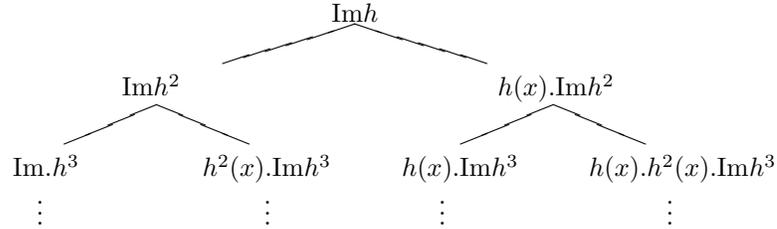
\begin{picture}(100,15)
	\drawline(0,0)(50,15)
	\drawline(100,0)(50,15)
\end{picture}\\
\hspace{10pt}$\ima h^2$\hspace{120pt}$h(x).\ima h^2$\\
\begin{picture}(70,15)
	\drawline(0,0)(35,15)
	\drawline(70,0)(35,15)
\end{picture}\hspace{80pt}\begin{picture}(70,15)
	\drawline(0,0)(35,15)
	\drawline(70,0)(35,15)
\end{picture}\\
\hspace{30pt}$\ima.h^3$ \hspace{40pt} $h^2(x).\ima h^3$ \hspace{20pt} $h(x).\ima h^3$ \hspace{20pt} $h(x).h^2(x).\ima h^3$\\
\hspace{0pt}\vdots \hspace{80pt} \vdots \hspace{60pt} \vdots \hspace{80pt} \vdots\hspace{0pt}
\end{center}
Consider the partial type $\Phi(x):=\{x\notin h^{-n}\ima h^{n+1},\ n\geq 1\}$. The sequence $(h^{-n}\ima h^{n+1})_{n \geq 1}$ is increasing, and each set $G\setminus h^{-n}\ima h^{n+1}$ is non-empty, so $\Phi$ is finitely consistent. Let $b$ be a realization of $\Phi$ in a saturated model. The graph $G(b)$ has $2^\omega$ consistent branches, whence $S_1(b)$ has cardinal $2^\omega$, a contradiction with $G$ being small.\end{proof}

\begin{cor}\label{cor}Let $G$ be a small group and $h$ a group homomorphism of $G$. There exists some integer $n$ such that $G$ equals $\ker h^{n}\cdot \ima h^n$.\end{cor}

\begin{proof}Take $n$ as in Proposition \ref{prop}, and set $f=h^n$. We have $\ima f^2=\ima f$, so for all $g\in G$ there exists $g'$ such that $f(g)=f^2(g')$. Hence $f(g f(g')^{-1})=1$ and $gf(g')^{-1}\in \ker f$, that is $g\in \ker f\cdot \ima f$.\end{proof}

It was shown in \cite{W2} that a definable injective homomorphism of a small group is surjective. Note that this follows again from Corollary \ref{cor}.\\

\section{Small skew fields}

Recall a result proved in \cite{W1} :

\begin{fact}\label{fact}An infinite small field is algebraically closed.\end{fact}

Let $D$ be an infinite small skew field. We begin by analysing elements of finite order.

\begin{lem}\label{lem}Let $a \in D$ an element of order $n<\omega$. Then $a$ is central in $D$.\end{lem}

\begin{proof}Either $D$ has zero characteristic, so $Z(D)$ is infinite, hence algebraically closed. But $Z(D)(a)$ is an extension of $Z(D)$ of degree $d\leq n$, whence $a\in Z(D)$.\\
Or $D$ has positive characteristic. Suppose that $a$ is not central, then \cite[Lemma 3.1.1 p.70]{Her} there exists $x$ in $D$ such that $xax^{-1}=a^i\neq a$. If $x$ has finite order, then all elements in the multiplicative group $\langle x,a\rangle$ have finite order. Hence $\langle x,a\rangle$ is commutative \cite[Lemma 3.1.3 p.72]{Her}, contradicting $xax^{-1}\neq a$. So $x$ has infinite order. Conjugating $m$ times by $x$, we get $x^m ax^{-m}=a^{i^m}$. But $a$ and $a^i$ have same order $n$, with $gcd(n,i)=1$. Put $m=\phi(n)$. By Fermat's Theorem, $i^m\equiv 1 [n]$, so $x^m$ and $a$ commute. Then $L=Z(C_D(a,x^m))$ is a definable infinite commutative subfield of $D$ which contains $a$. Let $L^{x}$ be $\{l\in L,\ x^{-1}lx=l\}$. This is a proper subfield of $L$. Moreover $1<[L^x(a):L^x]\leq n$. But $L^x$ is infinite as it contains $x$. By Fact \ref{fact}, it is algebraically closed and cannot have a proper extension of finite degree.\end{proof}

\begin{prop}\label{root}Every element of $D$ has a $n^\text{th}$ root for each $n\in\omega$.\end{prop}

\begin{proof}Let $a\in D$. If $a$ has infinite order, $Z(C_D(a))$ is an infinite commutative definable subfield of $D$. Hence it is algebraically closed, and $a$ has an $n^\text{th}$  root in $Z(C_D(a))$. Otherwise $a$ has finite order. According to Lemma \ref{lem} it is central in $D$. Let $x\in D$ have infinite order. Then $a\in Z(C_D(a,x))$, a commutative, infinite, definable, and thus algebraically closed field.\end{proof}

\begin{rmk} Note that since $D^\times$ is divisible, it has elements of arbitrary large finite order, which are central by Lemma \ref{lem}. Taking $D$ omega-saturated, we can suppose $Z(D)$ infinite.\end{rmk}

Let us now show that a small skew field is \textit{connected}, that is to say, has no definable proper subgroup of finite index.

\begin{prop}\label{con}$D$ is connected.\end{prop}

\begin{proof}Multiplicatively : By Proposition \ref{root}, $D^\times$ is divisible so has no subgroup of finite index. Additively : Let $H$ be a definable subgroup of $D^+$ of finite index $n$. In zero characteristic, $D^+$ is divisible, so $n=1$. In general, let $k$  be an infinite finitely generated subfield of $D$. Consider a finite intersection $G=\displaystyle\bigcap_{i\in I} d_iH$ of left translates of $H$ by elements in $k$ such that $G\cap k$ is minimal ; this exists by the chain condition. By minimality, $G\cap k=\displaystyle\bigcap_{d \in k} dH \cap k$, so $G\cap k$ is a left ideal of $k$. Furthermore, $G$ is a finite intersection of subgroups of finite index in $D^+$ ; it has therefore finite index in $D$. Thus $G\cap k$ has finite index in $D\cap k=k$, and cannot be trivial, so $G\cap k=k=H\cap k$. This holds for every infinite finitely generated $k$, whence $H=D$.\end{proof}

Now we look at elements of infinite order.

\begin{lem} $a\in D$ have infinite order. Then $C_D(a)=C_D(a^n)$ for all $n>0$.\end{lem}

\begin{proof}Clearly $C_D(a)\leq C_D(a^n)$. Consider $L=Z(C_D(a^n))$. It is algebraically closed by Fact \ref{fact}, but $L(a)$ is a finite commutative extension of $L$, whence $a\in L$ and $C_D(a^n)\leq C_D(a)$.\end{proof}

Now suppose that $D$ is not commutative. We shall look for a commutative centralizer $C$ and show that the dimension $[D:C]$ is finite. This will yield a contradiction.

\begin{lem}\label{plus}Let $a\in D$, $t\in D\setminus im(x\mapsto ax-xa)$ and $\varphi: x \mapsto t^{-1}\cdot(a x - x a)$. Then $D=im\varphi\oplus ker\varphi$. Moreover, if $k=dcl(a,t,\overline{x})$, where $\overline{x}$ is a finite tuple, then $k=im\varphi\cap k\oplus ker\varphi\cap k$.\end{lem}

\begin{proof}Let $K=ker\varphi=C_D(a)$. Put $I=im\varphi$; this is a right $K-$vector space, so $I\cap K=\{0\}$, since $1\in K\cap I$ is impossible by the choice of $t$. Consider the morphism
$$\tilde{\varphi}\ :\  \begin{tabular}{ccc} $D^+/K$ & $\longrightarrow$ &
$D^+/K$\\
$x$ & $\longmapsto$ & $\varphi(x)$
\end{tabular}$$
$\tilde{\varphi}$ is an embedding, and $D^+/K$ is small ; by Corollary 10, $\tilde{\varphi}$ is surjective and $D/K=\tilde{\varphi}(D/K)$, hence $D=I\oplus K$. Now, let $k=dcl(a,t,\overline{x})$ where $\overline{x}$ is a finite tuple of parameters in $D$. $I$ and $K$ are $k-$definable. For all $\alpha\in k$ there exists a unique couple $(\alpha_1,\alpha_2)\in I\times K$ such that $\alpha=\alpha_1+\alpha_2$, so $\alpha_1$ and $\alpha_2$ belong to $dcl(\alpha,a,t)\leq k$, that is to say $k=I\cap k\oplus K\cap k$.\end{proof}

\begin{lem}\label{onto}For every $a\notin Z(D)$, the map $\varphi_a:x\mapsto ax-xa$ is onto.\end{lem}

\begin{proof}Suppose $\varphi_a$ not surjective. Let $t\notin im\varphi_a$, and $k=dcl(t,a,\overline{x})$ be a non commutative subfield of $D$ for some finite tuple $\overline{x}$. Consider the morphism $$\varphi\ :\  \begin{tabular}{ccc} $D^+$ & $\longrightarrow$ &
$D^+$\\
$x$ & $\longmapsto$ & $t^{-1}.(a x - x a)$
\end{tabular}$$\\
Set $H = im\varphi$, and $K=C_D(a)=ker\varphi$. By Lemma \ref{plus} we get $k=H\cap k\oplus K\cap k$. Let $N=\bigcap_I a_iH$ be a finite intersection of left-translates of $H$ by elements in $k$, such that $N\cap k$ be minimal. We have $$N\cap k=\displaystyle\bigcap_{i\in I} a_iH\cap k =\displaystyle\bigcap_{d \in k} dH \cap k,$$ so $N\cap k$ is a left ideal. Moreover, $H\cap k$ is a right $K\cap k$ vector-space of codimension $1$. Then $N\cap k$ has codimension at most $n=\mid I\mid$. If $N\cap k=k$, then $H\cap k=k$, whence $K\cap k=\{0\}$, a contradiction. So $N\cap k$ is trivial and, $k$ is a $K\cap k$-vector space of dimension at most $n$. By \cite[Corollary 2 p.49]{Co} we get $[k:K\cap k]=[Z(k)(a):Z(k)]$. But $Z(k)=Z(C_D(k))\cap k$ with $Z(C_D(k))$ algebraically closed. Note that every element of $k$ commutes with $Z(C_D(k))$, so $a\in Z(k)$, which is absurd if we add $b\notin C_D(a)$ in $k$.\end{proof}

\begin{thm} A small field in non-zero characteristic is commutative.\end{thm}

\begin{proof}Let $a\in D$ be non-central, and let us show that $x\mapsto ax-xa$ is not surjective. Otherwise there exists $x$ such that $ax-xa=a$, hence $axa^{-1} = x+1$. We would then have $a^pxa^{-p} = x+p = x$, and $x\in C_D(a^p)\setminus C_D(a)$, a contradiction with Lemma \ref{lem}.\end{proof}

\section{Open problems}

\subsection{Zero characteristic}Note that we just use characteristic $p$ in proof of theorem 19 to show that there exist $a\notin Z(D)$ such that $x\mapsto ax-xa$ is not surjective. Thus questions 1 and 2 are equivalent :

\begin{que} Is a small skew field $D$ of zero characteristic commutative ?\end{que}

\begin{que} Is every $x\mapsto ax-xa$ surjective onto $D$ for $a\notin Z(D)$ ?\end{que}

\subsection{Weakly small fields} Weakly small structures have been introduced to give a common generalization of small and minimal structures. Minimal fields are known to be commutative.

\begin{defn} A structure $M$ is \emph{weakly small} if for all finite set of parameters $A$ in $M$, there are only countably many $1$-types over $A$.\end{defn}

\begin{que}Is a weakly small field algebraically closed ?\end{que}

\begin{que}Is a weakly small skew field commutative ?\end{que}

Note that a positive answer to question 3 implies a positive answer to question 4, as all the proves given still hold. In general, one can prove divisibility and connectivity of an infinite weakly small field.

\begin{prop} Every element in an infinite weakly small field $D$ has a $n^{th}$ root for all $n\in\omega$.\end{prop}

\begin{proof}Let $a\in D$. In zero characteristic, $Z(C_D(a))$ is an infinite definable commutative subfield of $D$, hence weakly small. According to \cite[Proposition 9]{W1}, every element in $Z(C_D(a))$ has a $n^{th}$ root. In positive characteristic, we can reason as in the proof of Lemma 12, and find $y$ with infinite order which commutes with $a$. Apply one more time \cite[Proposition 9]{W1} to $Z(C_D(a,y))$.\end{proof}

So $D^\times$ is divisible and the proof of Proposition \ref{con} still holds.

\begin{prop}An infinite weakly small field is connected.\end{prop}

\end{document}